\documentclass{amsart}
\usepackage{a4wide}
\usepackage[initials]{amsrefs}
\usepackage{url}
\usepackage[utf8]{inputenc}
\usepackage{amsmath,amssymb,amsthm,
graphicx,psfrag, epsfig, caption, float, color,multirow}
\usepackage{enumerate}
\usepackage{nicefrac}
\usepackage{multicol}

\newtheorem{theorem}{Theorem}
\newtheorem{lemma}{Lemma}
\newtheorem{corollary}{Corollary}

\newtheorem{definition}{Definition}
\newtheorem{conjecture}{Conjecture}

\newtheorem{innercustomlem}{Lemma}
\newenvironment{customlem}[1]
{\renewcommand\theinnercustomlem{#1}\innercustomlem}
  {\endinnercustomlem}

\theoremstyle{remark}
\newtheorem{remark}{Remark}
\title[Sums of four and more unit fractions
  and approximate parametrizations]{Sums of four and more unit fractions\\
  and approximate parametrizations}

\author[C. Elsholtz]{Christian Elsholtz}
\address[C. Elsholtz]{
Graz University of Technology, Institute of Analysis and Number Theory, Kopernikusgasse 24/II, 8010 Graz, Austria}
\email{elsholtz@math.tugraz.at}

\author[S. Planitzer]{Stefan Planitzer}
\address[S. Planitzer]{}
\email{stefan.planitzer@gmail.com}



\begin{document}

\begin{abstract}
We prove new upper bounds on the number of representations of rational numbers $\frac{m}{n}$ as a sum of $4$ unit fractions,
giving five different regions, depending on the size of $m$ in terms of $n$. 
In particular, we improve the most relevant cases, 
when $m$ is small, and when $m$ is close to $n$.
The improvements stem from not only studying complete
parametrizations of the set of solutions, but simplifying this set appropriately. 
Certain subsets of all parameters define the set of all solutions, 
up to applications of divisor functions, which has little impact on the upper bound of the number of solutions.  
These ``approximate parametrizations'' were the key point to enable computer programmes to filter through 
large number of equations and inequalities. Furthermore, this result leads to new upper bounds for the number of representations 
of rational numbers as sums of more than $4$ unit fractions.
\end{abstract}



\maketitle

\section{Introduction}

We consider the problem of representing an arbitrary positive rational number $\frac{m}{n}$ as a sum of $k$ unit fractions. This leads to Diophantine equations of the form
\begin{equation} \label{eq:generalUFeq}
\frac{m}{n}=\sum_{i=1}^k\frac{1}{a_i}.
\end{equation}
This equation has been studied from a variety of different view points,
we only mention results of Croot \cite{Croot}, Graham \cite{Graham},
Konyagin \cite{DoubleExponential}
and Martin \cite{DenseEgyptian}.

In this paper
we are interested in upper bounds for the number of solutions 
of~\eqref{eq:generalUFeq} in $(a_1, \ldots, a_k) \in \mathbb{N}^k$, 
in particular for fixed $m,n,k \in \mathbb{N}$,
where we consider the $a_i$ to be given in increasing order.

The most important special case of equation~\eqref{eq:generalUFeq} is when $m=4$ and $k=3$ which is linked to the famous Erd\H{o}s-Straus conjecture. This conjecture states that for any $n \geq 2$, the rational number $\frac{4}{n}$ has a representation as a sum of three unit fractions (see~\cite{Erdoes}). 
 
For a survey of recent results and for later use we borrow the following notation from~\cite{BrowningElsholtz}:
$$f_k(m,n)=\bigg|\bigg\{(a_1, \ldots, a_k) \in \mathbb{N}^k:a_1 \leq \ldots \leq a_k, \frac{m}{n}=\sum_{i=1}^k\frac{1}{a_i}\bigg\}\bigg|.$$
In case of the Erd\H{o}s-Straus equation with $n=p$ prime, 
Elsholtz and Tao~\cite{ElsholtzTao} proved that
\begin{equation} \label{eq:ET3Bound}
f_3(4,p) \ll_{\varepsilon} p^{\nicefrac{3}{5}+\varepsilon}.
\end{equation}
For general $m,n \in \mathbb{N}$ we have that
\begin{equation}\label{eq:BE3Bound}
f_3(m,n) \ll_{\varepsilon} n^{\varepsilon}\bigg(\frac{n}{m}\bigg)^{\nicefrac{2}{3}} \qquad \text{(Browning and Elsholtz~\cite{BrowningElsholtz})}
\end{equation}
and
\begin{equation}\label{eq:EP3Bound}
f_3(m,n) \ll_{\varepsilon} n^{\varepsilon} \bigg(\frac{n^3}{m^2}\bigg)^{\nicefrac{1}{5}} \qquad \text{(Elsholtz and Planitzer~\cite{ElsholtzPlanitzer}).} 
\end{equation}

Note that the upper bound in~\eqref{eq:EP3Bound} is stronger
than~\eqref{eq:BE3Bound} if $m\ll n^{\nicefrac{1}{4}}$. In particular the bound
in~\eqref{eq:EP3Bound} allows to deduce the Elsholtz-Tao exponent $3/5$ in~\eqref{eq:ET3Bound} for the Erd\H{o}s-Straus equation also for general denominators $n$.

Concerning sums of more than $3$ unit fractions the following upper bounds were proved in~\cite{BrowningElsholtz}):
for any $\varepsilon>0$
\begin{equation}\label{eq:BEkgeq4Bounds}
f_4(m,n)\ll_\varepsilon n^\varepsilon \Big\{\Big(\frac{n}{m}\Big)^{\frac{5}{3}}+
\frac{n^{\frac{4}{3}}}{m^{\frac{2}{3}}}\Big\},
\end{equation}
and for $k\geq 5$
\begin{equation}
f_k(m,n)\ll_\varepsilon  (kn)^\varepsilon
\Big(
\frac{k^{\frac{4}{3}}n^2}{m}\Big)^{(\nicefrac{5}{3})\cdot 2^{k-5}}.
\end{equation}

This was improved in~\cite{ElsholtzPlanitzer}:
\begin{equation}\label{eq:EPkgeq4Bounds}
f_4(m,n) \ll_{\varepsilon}
n^{\varepsilon}\bigg(\frac{n^{\nicefrac{4}{3}}}{m^{\nicefrac{2}{3}}}+\frac{n^{\nicefrac{28}{17}}}{m^{\nicefrac{8}{5}}}\bigg)
\end{equation}
and
\begin{equation}
f_k(m,n) \ll_{\varepsilon} (kn)^{\varepsilon}\bigg(\frac{k^{\nicefrac{4}{3}}n^2}{m}\bigg)^{(\nicefrac{28}{17})\cdot 2^{k-5}}, \qquad \text{for } k \geq 5.
\end{equation}

In case of $k=3$ the bounds in~(\ref{eq:ET3Bound}--\ref{eq:EP3Bound}) were
derived by analyzing suitable parametrizations of solutions of
equation~\eqref{eq:generalUFeq} together with an application of the classical
divisor bound. 
The method of Elsholtz and Tao~\cite{ElsholtzTao} leading 
to~\eqref{eq:ET3Bound} is possibly the limit of that method, and the same seems
to be true for the bound in~\eqref{eq:EP3Bound} (at least for constant $m$). 
However, we believe that these bounds are still quite far from the truth. 
Indeed, it was suggested by Heath-Brown to Elsholtz that even 
$f_3(m,n)=\mathcal{O}_{\varepsilon}(n^{\varepsilon})$ appears possible, as $n$ tends to infinity. 
More generally, and somewhat stronger, we think that it is also quite possible that the following conjecture holds true.

\begin{conjecture}
For $k,m$ fixed and $n \rightarrow \infty$ we have
$$f_k(m,n) \ll \exp\bigg(C_{m,k}\frac{\log n}{\log \log n}\bigg),$$
for a positive constant $C_{m,k}$ depending only on $m$ and $k$.
\end{conjecture}

The bounds in \eqref{eq:EPkgeq4Bounds} were derived via an application of a lifting procedure 
first introduced by Browning and Elsholtz~\cite{BrowningElsholtz}. The improvement in the bounds 
in~\eqref{eq:EPkgeq4Bounds} compared to the original bounds by Browning and Elsholtz comes 
from taking into account a small part of the information coming from parametrizations of solutions of~\eqref{eq:generalUFeq} for $k=4$ when lifting from $k=3$.

In this paper, our goal is to prove better upper bounds in the $k=4$ case 
\emph{directly} by using suitable parametrizations of the solutions and not by
lifting from the $k=3$ case. The problem with this approach is, that we want to
use a parametrization where the number of parameters increases exponentially
with $k$. The new method applied does not only use a suitable parametrization 
but in view of the increased complexity also has a computational part. In particular, we make heavy use of a computer algebra system to accomplish the following tasks:
\begin{itemize}
\item Find many \textit{defining sets}. By this we mean subsets of the parameters such that once they are fixed, we have at most of order $n^{\varepsilon}$ choices for the remaining parameters.
\item Find products of parameters which are small in terms of $n$ and such that the parameters appearing as factors may be partitioned into many defining sets.
\end{itemize} 
Note that what we call ``defining sets'' above are approximate parametrizations
in some sense. ``Defining sets'' are not in one-to-one correspondence with
solutions of equation~\eqref{eq:generalUFeq} as we would have with a full
parametrization. Nonetheless, fixing integer values for all parameters in 
a ``defining set'' allows for very few (in our sense $\mathcal{O}_{\varepsilon}(n^{\varepsilon})$) solutions for this equation instead of just a single one.    

Our main result is the following.

\begin{theorem} \label{thm:4UFThm}
For $m,n \in \mathbb{N}$ we have
$$f_4(m,n) \ll_{\varepsilon} n^{\varepsilon}\min \bigg\{ \frac{n^{\nicefrac{3}{2}  }}{m^{\nicefrac{3}{4} }}, \frac{n^{\nicefrac{8}{5}}}{m}\bigg\}.$$
\end{theorem}
Together with the two bounds in  \eqref{eq:BEkgeq4Bounds} and \eqref{eq:EPkgeq4Bounds} 
this gives:
\begin{corollary}
For $m,n \in \mathbb{N}$ we have
$$f_4(m,n) \ll_{\varepsilon} n^{\varepsilon}\min 
\bigg\{ \frac{n^{\nicefrac{3}{2}}}{m^{\nicefrac{3}{4}}}, 
\frac{n^{\nicefrac{8}{5}}}{m},
\bigg(\frac{n^{\nicefrac{4}{3}}}{m^{\nicefrac{2}{3}}}+\frac{n^{\nicefrac{28}{17}}}{m^{\nicefrac{8}{5}}}\bigg),
\bigg(\Big(\frac{n}{m}\Big)^{\frac{5}{3}}+\frac{n^{\frac{4}{3}}}{m^{\frac{2}{3}}}\bigg)
\bigg\}.$$
\end{corollary}
This new result shows that the analysis of the number sums of 4 and more 
unit fractions might be much more complicated than was previously known. 
\begin{remark}
In equation \eqref{eq:generalUFeq} one generally has that $a_1 \ll n, a_2 \ll n^2, a_3 \ll n^4$. Hence
 there are at most ${\mathcal O}(n^7)$ choices for $a_1,a_2$ and $a_3$, and then $a_4$ is unique, if it exists.
Hence $f_4(m,n)\ll n^{7}$ is a completely trivial upper bound.
However, fixing only $a_1$ and $a_2$ one sees that the number of pairs $(a_3,a_4)$ is bounded by a divisor function, 
(for details see e.g. \cite{ElsholtzPlanitzer}).
Hence $f_4(m,n)\ll n^{3+\varepsilon}$ is still a trivial upper bound.
The worst we would get from Theorem~\ref{thm:4UFThm}, when $m$ is small, 
would be an upper bound of order $n^{\nicefrac{3}{2}+\varepsilon}$.

Furthermore, if we compare the two upper bounds on 
$f_4(m,n)$ in Theorem~\ref{thm:4UFThm} with the previous bounds 
$n^{\varepsilon} \Big(\Big( \frac{n^{\nicefrac{28}{17}}}{m^{\nicefrac{8}{5}}}
+\frac{n^{\frac{4}{3}}}{m^{\frac{2}{3}}}\Big)$
in~\eqref{eq:EPkgeq4Bounds} and  
$n^{\varepsilon}\big(\big(\frac{n}{m}\big)^{\nicefrac{5}{3}}+\frac{n^{\nicefrac{4}{3}}}{m^{\nicefrac{2}{3}}}\big)$
in~\eqref{eq:BEkgeq4Bounds}, we see that each of these four bounds is best in some cases, and when splitting the contributions of the two parts in 
$\mathcal{O}\Bigl(\frac{n^{\nicefrac{28}{17}}}{m^{\nicefrac{8}{5}}} +\frac{n^{\nicefrac{4}{3}}}{m^{\nicefrac{2}{3}}} \Bigr)$, 
we see that there are even five different upper bounds involved:

To present these results in a uniform way we write exponents as $\nicefrac{\alpha}{30345}$, where $30345$ is the smallest integer
avoiding further fractions in the boundaries below.
For fractions $\frac{m}{n}$ with $m=n^{\nicefrac{\alpha}{30345}}$, where $\alpha$ is a real parameter in 
$0 \leq \alpha \leq 30345$, the following holds, (omitting the $n^{\varepsilon}$ factor):
\begin{itemize}
\item $0 \leq \alpha \leq 5250$: the upper bound of order $\frac{n^{\nicefrac{3}{2}}}{m^{\nicefrac{3}{4}}}$ from Theorem~\ref{thm:4UFThm} is the sharpest one.
\item $5250 \leq \alpha \leq 8925$: the bound 
$\frac{n^{\nicefrac{28}{17}}}{m^{\nicefrac{8}{5}}}$
from~\eqref{eq:EPkgeq4Bounds} gives the best bound.
\item $8925 \leq \alpha  \leq 10115$: in this range the bound
  $(\frac{n}{m})^{\nicefrac{5}{3}}$ from \eqref{eq:BEkgeq4Bounds} yields the
  lowest upper bound. (Note that $10115/30345=1/3$.)
\item $10115\leq \alpha\leq 10200$. In this very
  small region the bound 
 $\frac{n^{\nicefrac{4}{3}}}{m^{\nicefrac{2}{3}}}$
 from \eqref{eq:BEkgeq4Bounds} gives the best bound.
\item $10200 \leq \alpha \leq 24276$. (Note that $24276/30345=4/5)$. In this region
the bound is also $\frac{n^{\nicefrac{4}{3}}}{m^{\nicefrac{2}{3}}}$, but this
time it comes from
 \eqref{eq:BEkgeq4Bounds} and \eqref{eq:EPkgeq4Bounds}.
\item $24276 \leq \alpha \leq 30345$: the second bound in Theorem~\ref{thm:4UFThm} 
which is of order $\frac{n^{\nicefrac{8}{5}}}{m}$ gives the best bound. 
\end{itemize}

At the points of transition, i.e. $\alpha \in \{5250,8925,10115, 10200,24276\}$, in these inequalities the corresponding upper bounds are equally sharp.
\end{remark}
We summarize this in the following corollary.
\begin{corollary}
For $m,n \in \mathbb{N}$ we have
$$f_4(m,n) \ll_{\varepsilon} 
\begin{cases}
n^{\varepsilon} \frac{n^{\nicefrac{3}{2}}}{m^{\nicefrac{3}{4}}}& 
\text{ if }m  \ll n^{\nicefrac{50}{289}},\\ 
n^{\varepsilon} \frac{n^{\nicefrac{28}{17}}}{m^{\nicefrac{8}{5}}}&
\text{ if } n^{\nicefrac{50}{289}}\ll m  \ll n^{\nicefrac{5}{17}},\\ 
n^{\varepsilon} \Big(\frac{n}{m}\Big)^{\nicefrac{5}{3}}&
\text{ if } n^{\nicefrac{5}{17}}\ll m  \ll n^{\nicefrac{1}{3}},\\ 
n^{\varepsilon}\frac{n^{\nicefrac{4}{3}}}{m^{\nicefrac{2}{3}}}& 
\text{ if }  n^{\nicefrac{1}{3}} \ll m  \ll n^{\nicefrac{4}{5}},\\  
n^{\varepsilon} \frac{n^{\nicefrac{8}{5}}}{m}& \text{ if }
n^{\nicefrac{4}{5}} \ll m  \ll n .
\end{cases}
$$
\end{corollary}

\begin{figure}
\begin{minipage}{1.0\textwidth}
\centering
\includegraphics[width=\textwidth]{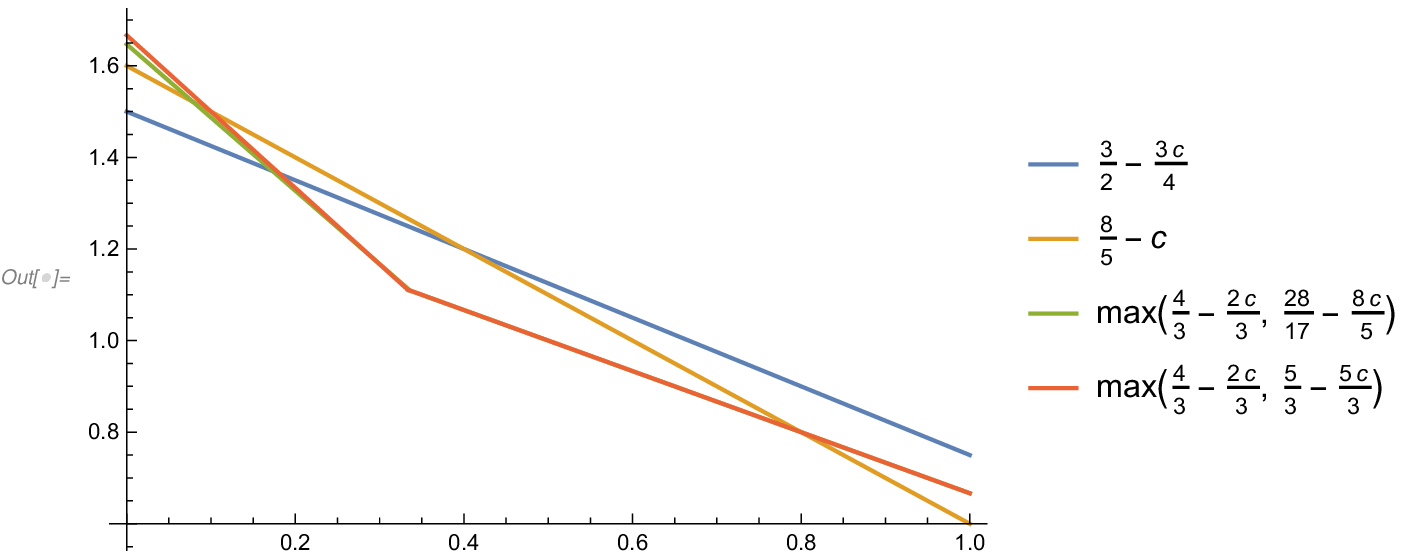}
\caption{The full range with $0 \leq c=\alpha/30345 \leq 1$.}
\label{fig1}
\end{minipage}
\hfill
Recall that the new bounds are the blue line 
(strongest on the left hand side), and beige, 
(strongest on the right hand side of the graph).

\begin{minipage}{1.0\textwidth}
\centering
\includegraphics[width=\textwidth]{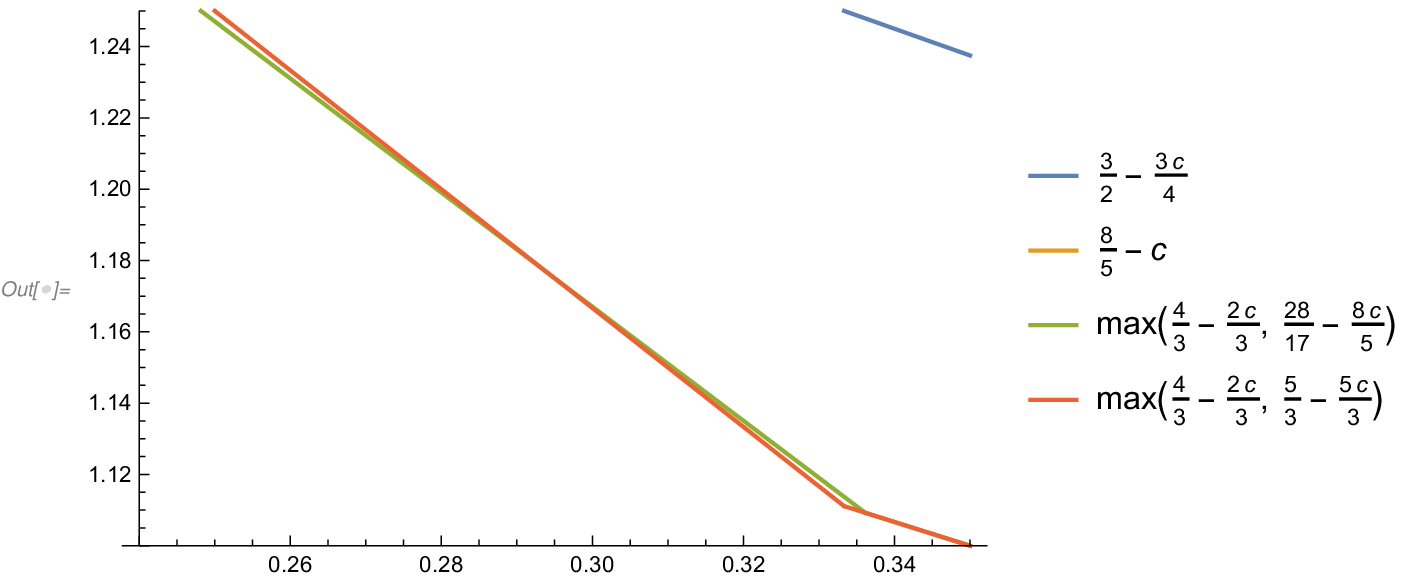}
\caption{The region $0.24 \leq c\leq 0.35$ enlarged, to see the
  crossing of almost parallel lines.
}
\label{fig2}
\end{minipage}
\end{figure}

\begin{remark}
In the proof of Theorem~\ref{thm:4UFThm} we give a method for constructing all representations 
of a rational number $\frac{m}{n}$ as a sum of four unit fractions. 
Along the same lines as the proof of the corresponding result on sums of three unit fractions in~\cite{ElsholtzPlanitzer}, 
it can be shown that there exists an algorithm with expected running time of order 
$$
 n^{\varepsilon}\min 
\bigg\{ \frac{n^{\nicefrac{3}{2}}}{m^{\nicefrac{3}{4}}}, 
\frac{n^{\nicefrac{8}{5}}}{m},
\bigg(\frac{n^{\nicefrac{4}{3}}}{m^{\nicefrac{2}{3}}}+\frac{n^{\nicefrac{28}{17}}}{m^{\nicefrac{8}{5}}}\bigg),
\bigg(\Big(\frac{n}{m}\Big)^{\frac{5}{3}}+\frac{n^{\frac{4}{3}}}{m^{\frac{2}{3}}}\bigg)
\bigg\}$$
listing these solutions. In particular, we can decide within the same time
constraints whether or not the rational number $\frac{m}{n}$ has a
representation of this form. A precise formulation of this result would make
use of the complexity of factorizations. For details we refer to \cite{ElsholtzPlanitzer}.
\end{remark}

Again the bound on sums of four unit fractions can be lifted to upper bounds for $k>4$.

\begin{theorem} \label{thm:kgeq5UFThm}
For $m,n \in \mathbb{N}$ and $k\geq 5$ we have
$$f_k(m,n) \ll_\varepsilon (kn)^{\varepsilon}\bigg(\frac{k^{\nicefrac{4}{3}}n^2}{m}\bigg)^{(\nicefrac{8}{5})\cdot 2^{k-5}}.$$
\end{theorem}

Note that the improvement in the upper bound in Theorem~\ref{thm:kgeq5UFThm} concerns the constant $\frac{8}{5}$ in the exponent. If we compare the result with the bounds in~\eqref{eq:EPkgeq4Bounds}, we see that, depending on $k$, the difference in the corresponding exponents of $n$ is $\frac{4}{85}\cdot 2^{k-4}$.

The results in Theorem~\ref{thm:kgeq5UFThm} immediately improve several upper bounds for the special case of representing $1$ as a sum of unit fractions. Some of these results are mentioned in~\cite{BrowningElsholtz} with improved upper bounds in~\cite{ElsholtzPlanitzer}. Here we just reformulate~\cite{ElsholtzPlanitzer}*{Corollary 3} by giving the improved upper bounds one gets by using Theorem~\ref{thm:kgeq5UFThm}. The proof is the same as in~\cite{ElsholtzPlanitzer} and~\cite{BrowningElsholtz} after plugging in the new bound.

\begin{corollary}
\begin{enumerate}
\item For any $\varepsilon>0$, we have that
$$f_k(1,1)\ll_{\varepsilon} k^{(\nicefrac{2}{15})\cdot 2^{k-1} + \varepsilon}.$$
\item Let $(u_n)_{n \in \mathbb{N}}$ be the sequence recursively defined by $u_0=1$ and $u_{n+1}=u_n(u_n+1)$ and set $c_0=\lim_{n\rightarrow \infty}u_n^{2^{-n}}$. Then for $\varepsilon > 0$ and $k\geq k(\varepsilon)$ we have
$$f_k(1,1) < c_0^{(\nicefrac{2}{5}+\varepsilon)2^{k-1}}.$$
\item For $\varepsilon>0$ and $k\geq k(\varepsilon)$ the number of positive integer solutions of the equation
$$1=\sum_{i=1}^k\frac{1}{a_i}+\frac{1}{\prod_{i=1}^ka_i}$$
is bounded from above by $c_0^{(\nicefrac{2}{5}+\varepsilon)2^k}$.
\end{enumerate}
\end{corollary}
\begin{remark}
  The sequence $u_n$, starting with $1,2,6,42,1806, \ldots$ is listed as A007018 in the
  online encyclopedia of integer seqeuences (OEIS), and is a shifted copy
  of the well known
  Sylvester sequence (A000058 of the OEIS): $2,3,7,43,1807, \ldots$ It is
  known that the limit $c_0=\lim_{n\rightarrow
    \infty}u_n^{2^{-n}}=1.5979102\ldots$ exists and is irrational, for details see
  \cite{Aho-Sloane} and \cite{Wagner-Ziegler}.
  Graham, Knuth and Patashnik
  \cite[Exercise 4.37]{Graham-Knuth-Patashnik} sketch a proof of (in our notation) $u_n= \lfloor
  c_0^{2^n}- \frac{1}{2} \rfloor$.
  The existence of the limit can be proved directly, as it follows inductively
  that $u_n \leq \frac{2^{2^n}}{2}$, so that the sequence $q_n:=(u_n)^{1/2^n}$ is
bounded from above by $2$, and $u_{n+1}\geq u_n^2$ implies that
$(u_{n+1})^{1/2^{n+1}}\geq    (u_n)^{1/2^n}$, so that the sequence of the $q_n$
is also monotonically increasing.

\end{remark}

At the end of this introduction we want to comment on the most important aspects of the notation used in the following. The letters $\mathbb{N}$ and $\mathbb{P}$, as usual, denote the sets of positive integers and positive primes. The function $d(n)$ denotes the number of positive divisors of $n$. By $\nu_p(n)$, $p \in \mathbb{P}$, we denote the $p$-adic valuation of $n$, i.e.~the highest power of $p$ dividing $n$.  We use the symbols $\ll$ and $\mathcal{O}$ in the contexts of the well known Vinogradov- and Landau-notations. Dependencies of the implied constants on additional parameters will be indicated by a subscript.

\section{Patterns and parameters} \label{sec:param}

In this section, we introduce a method of parametrization for solutions of equation~\eqref{eq:generalUFeq} which is based on what we will call relative greatest common divisors and patterns. This type of parametrization has been used before in connection with sums of unit fractions. Elsholtz first used relative greatest common divisors as described below in~\cite{Elsholtz} while patterns played a role in proving results in~\cite{ElsholtzPlanitzer}. For a more thorough introduction to this method and for some historical comments see~\cite{Elsholtz} and~\cite{ElsholtzPlanitzer}.

We start by writing the denominators of the unit fractions on the right hand side of equation~\eqref{eq:generalUFeq} as $a_i=n_it_i$, where $n_i = \gcd(a_i,n)$. We note that by definition $\gcd\big(t_i,\frac{n}{n_i}\big)=1$ and for given $(a_1,\ldots, a_k)\in \mathbb{N}^k$ we call $(n_1, \ldots, n_k)\in \mathbb{N}^k$ the \textit{pattern} of the solution. To bound the number of patterns for given $n \in \mathbb{N}$, we make use of the classical divisor bound which was also one of the main ingredients in Elsholtz and Tao's proof of an upper bound for $f_3(4,p)$ in~\cite{ElsholtzTao}. We will use it in the following form (see~\cite{HardyWright}*{Theorem 315}).

\begin{customlem}A[Classical divisor bound] \label{lem:divisorBound}
Let $d(n)=\sum_{d|n}1$ be the number of positive divisors of an integer $n$. Then for any $\varepsilon >0$, we have
$$d(n) \ll_{\varepsilon} n^{\varepsilon}.$$
\end{customlem}

When trying to find upper bounds on $f_4(m,n)$, we can consider the pattern of the solutions to be fixed, since the upper bound we will establish is independent of the pattern. Lemma~\ref{lem:divisorBound} tells us, that we have at most $\mathcal{O}_{\varepsilon}(n^{\epsilon})$ such patterns and when looking at the result in Theorem~\ref{thm:4UFThm} we see that an additional factor of $n^{\varepsilon}$ does not change the upper bound there. Hence from now on we consider the pattern $(n_1,n_2,n_3,n_4)$ to be fixed. 

Note that the trivial upper bound for the number of patterns would rather be of order $n^{4\varepsilon}$ and to get the above bound we need to redefine $\varepsilon$. Also below we will often apply the divisor bound several times in a row to conclude, that there are at most of order $n^{\varepsilon}$ choices for some parameters. In any such situation this upper bound is achieved after possibly redefining $\varepsilon$, and we will not explicitly state this henceforth.  

Next we set $I=\{1, \ldots, k\}$ to be the index set and write the factors $t_i$ as a product of what we want to call relative greatest common divisors denoted by $x_J$, $J=\{i_1, \ldots, i_{|J|}\}  \subset I$. Here we recursively define these relative greatest common divisors $x_J$ as follows:
$$x_I=\gcd(t_1, \ldots, t_k) \text{ and } x_J =\frac{\gcd(t_{i_1}, \ldots, t_{i_{|J|}})}{\prod_{J \subsetneq K}x_K} \text{ for } J \subsetneq I.$$
With this definition, we have
$$t_i=\prod_{\substack{J \subset I \\ i \in J}}x_J \text{ for } 1 \leq i \leq k$$
and it is easy to see that 
\begin{equation} \label{eq:RGCDCoprimalityeqn}
\gcd(x_J,x_K)=1 \text{ whenever } J \nsubseteq K \text{ and } K \nsubseteq J.
\end{equation}
See e.g.~\cite{ElsholtzPlanitzer} for a short proof of the last statement.

To keep things readable, and since in the cases we use it no ambiguity will arise, below we will often resort to the following simplified notation. If $J=\{i_1, \ldots, i_{|J|}\}$ and the $i_j$ are given in increasing order, then we write
$$x_J=x_{i_1i_2\ldots i_{|J|}}.$$

We now apply this parametrization and patterns in the special case of sums of $4$ unit fractions, i.e. equation~\eqref{eq:generalUFeq} with $k=4$:
$$\frac{m}{n}=\frac{1}{a_1}+ \cdots +\frac{1}{a_4},$$
where $a_1 \leq \ldots \leq a_4$. Let $(n_1,\ldots, n_4)$ be our fixed pattern and thus $a_i=n_it_i$ for $1 \leq i \leq 4$.

We use relative greatest common divisors and the fixed pattern to write
\begin{equation}\label{eq:initial4UFeq}
\begin{split}
\frac{m}{n}=&\frac{1}{n_1x_1x_{12}x_{13}x_{14}x_{123}x_{124}x_{134}x_{1234}}+\frac{1}{n_2x_2x_{12}x_{23}x_{24}x_{123}x_{124}x_{234}x_{1234}}+ \\
&\frac{1}{n_3x_3x_{13}x_{23}x_{34}x_{123}x_{134}x_{234}x_{1234}}+\frac{1}{n_4x_4x_{14}x_{24}x_{34}x_{124}x_{134}x_{234}x_{1234}}.
\end{split}
\end{equation}
Next we multiply the last equation by $n$ and the least common denominator of the unit fractions on the right hand side. Note that after doing so, the variable $x_i$, for $1\leq i\leq 4$, appears in exactly three of the four summands on the right hand side and in the product on the left hand side. This means that also the fourth summand on the right hand side, of which $x_i$ is not a factor, has to be divisible by $x_i$. This factor is of the form
$$\frac{n}{n_i}\prod_{\substack{J \subset I \\ i \not\in J}}x_J,$$
where we use the set-index notation for
convenience. By~\eqref{eq:RGCDCoprimalityeqn}, $x_i$ is coprime to
$\prod_{\substack{J \subset I \\ i \not\in J}}x_J$. Furthermore, by the
definition of a pattern, we also have $\gcd\big(x_i,\frac{n}{n_i}\big)=1$, which leaves $x_i=1$ for $1 \leq i \leq 4$. With this simplification we get
\begin{equation}\label{eq:initial4UFeqmult}
\begin{split}
mx_{12}x_{13}x_{14}x_{23}x_{24}x_{34}x_{123}x_{124}x_{134}x_{234}x_{1234}=&\frac{n}{n_1}x_{23}x_{24}x_{34}x_{234}+\frac{n}{n_2}x_{13}x_{14}x_{34}x_{134}+\\
&\frac{n}{n_3}x_{12}x_{14}x_{24}x_{124}+\frac{n}{n_4}x_{12}x_{13}x_{23}x_{123}.
\end{split}
\end{equation}

We introduce the parameters $d_{\{i,j\}}=d_{ij}=\gcd\big( \frac{n}{n_i},\frac{n}{n_j}\big)$ for $1\leq i<j\leq 4$ and $d_{\{i,j,k\}}=d_{ijk}=\gcd\big( \frac{n}{n_i},\frac{n}{n_j},\frac{n}{n_k} \big)$ for $1 \leq i<j<k \leq 4$ and we note that they are fixed by the pattern $(n_1,\ldots, n_4)$. Furthermore, again by definition of a pattern, we have that $d_{ij}$ is coprime to all relative greatest common divisors with an $i$ or a $j$ in the index. The same holds true for $d_{ijk}$ and relative greatest common divisors with an $i$, $j$ or $k$ in the index. 

In~\cite{BrowningElsholtz},~\cite{ElsholtzPlanitzer} and~\cite{ElsholtzTao} it turned out to be useful to consider divisibility relations in the equation corresponding to~\eqref{eq:initial4UFeqmult} in the three unit fractions case. We will also do this and define the following integer parameters:

\begin{equation} \label{eq:zijdef}
\begin{split}
z_{23}&=\frac{\frac{n}{n_2d_{23}}x_{13}x_{34}x_{134}+\frac{n}{n_3d_{23}}x_{12}x_{24}x_{124}}{x_{23}}\\
z_{34}&=\frac{\frac{n}{n_3d_{34}}x_{14}x_{24}x_{124}+\frac{n}{n_4d_{34}}x_{13}x_{23}x_{123}}{x_{34}}
\end{split}
\end{equation} 
\begin{equation} \label{eq:zijkdef}
\begin{split}
z_{123}&=\frac{\frac{n}{n_1d_{123}}x_{23}x_{24}x_{34}x_{234}+\frac{n}{n_2d_{123}}x_{13}x_{14}x_{34}x_{134}+\frac{n}{n_3d_{123}}x_{12}x_{14}x_{24}x_{124}}{x_{12}x_{13}x_{23}x_{123}}\\
z_{134}&=\frac{\frac{n}{n_1d_{134}}x_{23}x_{24}x_{34}x_{234}+\frac{n}{n_3d_{134}}x_{12}x_{14}x_{24}x_{124}+\frac{n}{n_4d_{134}}x_{12}x_{13}x_{23}x_{123}}{x_{13}x_{14}x_{34}x_{134}}\\
z_{234}&=\frac{\frac{n}{n_2d_{234}}x_{13}x_{14}x_{34}x_{134}+\frac{n}{n_3d_{234}}x_{12}x_{14}x_{24}x_{124}+\frac{n}{n_4d_{234}}x_{12}x_{13}x_{23}x_{123}}{x_{23}x_{24}x_{34}x_{234}}.
\end{split}
\end{equation}

In the following we will only use the parameters $z_J$ defined above. For a general definition of $z_J$, $J \subset \{1, \ldots, 4\}$, $2 \leq |J| \leq 3$, see Section~\ref{sec:comp}.

\section{Defining sets for sums of four unit fractions} \label{sec:definingSets}

In this section, we will determine several defining sets for sums of four unit fractions. We define these sets in the following way.

\begin{definition}
Let $m,n \in \mathbb{N}$, $(n_1,\ldots, n_4)\in \mathbb{N}^4$ be a fixed pattern, $I=\{1, \ldots, 4\}$ and $\mathcal{P}=\mathcal{X}\cup \mathcal{Z}$, where
\begin{equation} \label{def:setsXandZ}
\mathcal{X}=\big\{x_J: J\subset I, |J| \geq 2\big\} \text{ and } \mathcal{Z}=\{z_J: J \subset I, 2\leq |J| \leq 3\}
\end{equation}
are the sets of parameters introduced in Section~\ref{sec:param}. We call a set $S\subset P$ a (four unit fractions) defining set, if assigning a positive integer value to every parameter in $S$ allows for at most $\mathcal{O}_{\varepsilon}(n^{\varepsilon})$ positive integer assignments to variables in $\mathcal{X}\backslash S$ such that
$$\frac{m}{n}=\sum_{i=1}^4\frac{1}{n_i\prod_{\substack{J \subset I \\ i \in J \\ |J| \geq 2}}x_J}.$$ 
\end{definition}

Note that the idea behind the ``defining sets'' was already applied in~\cite{ElsholtzTao}*{Section 3} and~\cite{ElsholtzPlanitzer} when dealing with sums of three unit fractions (in~\cite{ElsholtzPlanitzer} actually also in the four unit fractions case, but to a very limited extent). Since the larger number of parameters in the four unit fractions case leads to a lot more possibilities for defining sets than we had when dealing with sums of three unit fractions, it seems impractical to determine these sets by hand. In Section~\ref{sec:comp}, we describe how we computed many defining sets via a structured approach using a computer algebra system. Any of these new defining sets can easily be verified by hand. In particular, we will prove the following Lemma, which covers only the defining sets used to prove Theorem~\ref{thm:4UFThm}.

\begin{lemma} \label{lem:definingSetLemma}
The following sets are four unit fractions defining sets:

\begin{multicols}{2}
\begin{enumerate}
\item $\{z_{23}, z_{234}\}$,
\item $\{z_{234},x_{23},x_{24}\}$,
\item $\{z_{234},x_{23},x_{234}\}$,
\item $\{z_{34},x_{12},x_{123},x_{124},x_{1234}\}$,
\item $\{x_{12},x_{13},x_{24},x_{34},x_{123},x_{124},x_{134},x_{1234}\}$,
\item $\{x_{12},x_{13},x_{14},x_{23},x_{123},x_{124},x_{134},x_{234},x_{1234}\}$.
\end{enumerate}
\end{multicols}
\end{lemma}

\begin{proof}

With the help of equations~(\ref{eq:initial4UFeqmult}-\ref{eq:zijkdef}) we derive the following set of equations:
\begin{align}
mx_{14}x_{24}x_{34}x_{124}x_{134}x_{234}x_{1234}&=d_{123}z_{123}+\frac{n}{n_4}\label{eq:systemeq1}, \\
mx_{12}x_{13}x_{14}x_{123}x_{124}x_{134}x_{1234}&=d_{234}z_{234}+\frac{n}{n_1}\label{eq:systemeq2}, \\
z_{23}x_{23}&=\frac{n}{n_2d_{23}}x_{13}x_{34}x_{134}+\frac{n}{n_3d_{23}}x_{12}x_{24}x_{124}\label{eq:systemeqn1},\\
z_{34}x_{34}&=\frac{n}{n_3d_{34}}x_{14}x_{24}x_{124}+\frac{n}{n_4d_{34}}x_{13}x_{23}x_{123}\label{eq:systemeq3}, \\
z_{234}x_{24}x_{34}x_{234}&=\frac{d_{23}}{d_{234}}x_{14}z_{23}+\frac{n}{n_4d_{234}}x_{12}x_{13}x_{123}\label{eq:systemeq4}, \\
z_{234}x_{23}x_{24}x_{234}&=\frac{d_{34}}{d_{234}}x_{12}z_{34}+\frac{n}{n_2d_{234}}x_{13}x_{14}x_{134}\label{eq:systemeq5}, \\
z_{134}z_{234} &= \frac{n^2}{n_1n_2d_{134}d_{234}}+\frac{n^2d_{34}}{d_{134}d_{234}}z_{34}x_{12}^2x_{123}x_{124}x_{1234} \label{eq:systemeq6}.
\end{align}
The method of proof will be as follows. We show that fixing positive integer values for the parameters in the sets in the statement of the lemma fixes the right hand side of at least one of the equations~(\ref{eq:systemeq1}--\ref{eq:systemeq6}). From the divisor bound in Lemma~\ref{lem:divisorBound} we may then deduce that we have at most of order $n^{\epsilon}$ choices for the variables on the left hand side of the corresponding equation. For any of these choices of new parameters we may then iterate the argument.

Here we note that the right hand sides of equations~(\ref{eq:systemeq1}--\ref{eq:systemeq6}) are at most of polynomial sizes in $n$. By definition, the parameters $d_J$, $J \subset \{1,\ldots, 4\}$, $2\leq |J|\leq 3$, are bounded from above by $n$. If we have a look at the definition of the parameters in the set $\mathcal{Z}$ in~\eqref{def:setsXandZ}, we see that they are certainly of size at most polynomial in $n$, if the same is true for the parameters in the set $\mathcal{X}$. To see that the relative greatest common divisors in $\mathcal{X}$ are of size at most polynomial in $n$, we use the fact that any of them is a factor of at least two of the denominators $a_i$, $1 \leq i \leq 4$. In particular, if we have $\frac{m}{n}=\frac{1}{a_1}+\cdots +\frac{1}{a_4}$ with $0<a_1 \leq \ldots \leq a_4$, then
$$\frac{m}{n}\leq \frac{4}{a_1} \text{ and } a_1 \leq \frac{4n}{m}.$$
With a similar argument we get
$$\frac{m}{n}-\frac{1}{a_1}=\frac{ma_1-n}{na_1} \leq \frac{3}{a_2} \text{ and } a_2 \leq 3na_1 \leq \frac{12n^2}{m}.$$
Finally we derive from the last two inequalities
$$\frac{m}{n}-\frac{1}{a_1}-\frac{1}{a_2}=\frac{ma_1a_2-na_1-na_2}{na_1a_2}\leq \frac{2}{a_3} \text{ and } a_3 \leq 2na_1a_2 \leq \frac{96n^4}{m^2}.$$
We now go through all defining sets in the statement of the Lemma.
\begin{enumerate}
\item Once we fix positive integer values for $z_{23}$ and $z_{234}$, we deduce from equation~\eqref{eq:systemeq2}, that we have at most of order $n^{\varepsilon}$ may choices for all relative greatest common divisors with a `$1$' in the index. Equation~\eqref{eq:systemeq4} then implies the same for the variables $x_{24}$, $x_{34}$ and $x_{234}$. Finally, the missing variable $x_{23}$ is uniquely determined by~\eqref{eq:initial4UFeqmult}.

\item We now consider $z_{234}$, $x_{23}$ and $x_{24}$ to be fixed. Again we have at most of order $n^{\varepsilon}$ choices for all relative greatest common divisors with a `$1$' in the index by~\eqref{eq:systemeq2}. Now the same holds true for the parameters $z_{34}$ and $x_{34}$ by equation~\eqref{eq:systemeq3}. Via equation~\eqref{eq:systemeq5} we deduce that we have at most of order $n^{\varepsilon}$ choices for the missing parameters $x_{23}$, $x_{24}$ and $x_{234}$.  

\item Having assigned positive integer values to the parameters $z_{234}$, $x_{23}$ and $x_{234}$, we again use equation~\eqref{eq:systemeq4} to deduce, that we have at most of order $n^{\varepsilon}$ many choices for all parameters with a `1' in the index. Now only assignments for the parameters $x_{24}$ and $x_{34}$ are missing. 

To see that we also have at most of order $n^{\varepsilon}$ many choices for these two parameters, we will apply a method of factoring equation~\eqref{eq:initial4UFeqmult} which was already used by Browning and Elsholtz~\cite{BrowningElsholtz}. As two of the five terms of equation~\eqref{eq:initial4UFeqmult} contain the factor $x_{24}x_{34}$, it may be rewritten in the form
$$C_1x_{24}x_{34}=C_2x_{24}+C_3x_{34}+C_4$$
and further
$$(C_1x_{24}-C_3)(C_1x_{34}-C_2)=C_1C_4+C_2C_3,$$
where the constants $C_i$, $1 \leq i \leq 4$, depend only on relative greatest common divisors $x_J$ which are known. The last equation implies that also in this case, for the remaining parameters $x_{24}$ and $x_{34}$ we have at most of order $n^{\varepsilon}$ many choices. \label{method:factoring}

\item In the case of $z_{34}$, $x_{12}$, $x_{123}$, $x_{124}$ and $x_{1234}$ being fixed, we see that we have at most of order $n^{\varepsilon}$ choices for the parameters $z_{134}$ and $z_{234}$ by equation~\eqref{eq:systemeq6}. From equations~\eqref{eq:systemeq1} and~\eqref{eq:systemeq2} we now see that we have at most of order $n^{\varepsilon}$ choices for $x_{13}$, $x_{14}$, $x_{24}$, $x_{34}$, $x_{134}$, and $x_{234}$ . This last parameter, $x_{23}$, is finally uniquely determined by~\eqref{eq:initial4UFeqmult}. 

\item If all the parameters $x_{12}$, $x_{13}$, $x_{24}$, $x_{34}$, $x_{123}$, $x_{124}$, $x_{134}$ and $x_{1234}$ are fixed, we see from equation~\eqref{eq:systemeqn1} that we have of order $n^{\varepsilon}$ choices for the parameter $x_{23}$. Now only the parameters $x_{14}$ and $x_{234}$ are missing. At this point we again use that equation~\eqref{eq:initial4UFeqmult} factors. Indeed, we may rearrange this equation to take the form
$$C_1x_{14}x_{234}=C_2x_{14}+C_3x_{234}+C_4,$$
where $C_1$, $C_2$, $C_3$ and $C_4$ are integer constants. This equation factors as in point~\eqref{method:factoring}, which leads to at most $\mathcal{O}_{\varepsilon}(n^{\varepsilon})$ choices for $x_{14}$ and $x_{234}$. \label{method:factoring2}

\item We now deal with the case when $x_{12}$, $x_{13}$, $x_{14}$, $x_{23}$, $x_{123}$, $x_{124}$, $x_{134}$, $x_{234}$ and $x_{1234}$ are all fixed. Note that only the two variables $x_{24}$ and $x_{34}$ are missing out. We already proved in point~\eqref{method:factoring} that in this case we have at most of order $n^{\varepsilon}$ many choices for these two parameters.

\end{enumerate}

\end{proof}

\section{Upper bounds on sums of 4 unit fractions} \label{sec:UF4}

In this section, we apply the parametrization introduced in Section~\ref{sec:param} and defining sets from Section~\ref{sec:definingSets} together with ideas from~\cite{ElsholtzTao}*{Section 3} and~\cite{ElsholtzPlanitzer} to prove Theorem~\ref{thm:4UFThm}. Recall, that with a fixed pattern all variables $n_i$, $d_{ij}$ and $d_{ijk}$ are fixed for $1\leq i,j,k \leq 4$ and we have $\mathcal{O}_{\epsilon}(n^{\epsilon})$ patterns altogether.

We now use the fact that the denominators $a_i=n_it_i$ are given in increasing order. The inequalities $a_2 \leq a_3$ and $a_3\leq a_4$ may be rewritten as
$$x_{12}x_{24}x_{124} \leq \frac{n_3}{n_2}x_{13}x_{34}x_{134}, \quad x_{13}x_{23}x_{123}\leq \frac{n_4}{n_3}x_{14}x_{24}x_{124},$$
by just plugging in the corresponding products of relative greatest common divisors for the $t_i$, $2\leq i\leq 4$. Combining these last inequalities with three of the equations in \eqref{eq:zijdef} and \eqref{eq:zijkdef} yields
\begin{align}
z_{23}x_{23} &\leq \frac{2n}{n_2d_{23}}x_{13}x_{34}x_{134}\label{eq:x23ineq},\\
z_{34}x_{34} &\leq \frac{2n}{n_3d_{34}}x_{14}x_{24}x_{124}\label{eq:x34ineq},\\
z_{234}x_{23}x_{24}x_{234} &\leq \frac{3n}{n_2d_{234}}x_{13}x_{14}x_{134}\label{eq:x234ineq}.
\end{align}
Furthermore, since the denominators $a_i$ are given in ascending order, we deduce from $\frac{m}{n}=\frac{1}{n_1t_1}+\frac{1}{n_2t_2}+\frac{1}{n_3t_3}+\frac{1}{n_4t_4}$ that $\frac{m}{n}\leq \frac{4}{n_1t_1}$ and hence
\begin{equation}\label{eq:t1ineq}
t_1=x_{12}x_{13}x_{14}x_{123}x_{124}x_{134}x_{1234} \leq \frac{4n}{n_1m}.
\end{equation}

We now prove the two upper bounds in Theorem~\ref{thm:4UFThm} separately. We start with the upper bound of order $n^{\varepsilon}\big(\frac{n^{\nicefrac{3}{2}}}{m^{\nicefrac{3}{4}}}\big)$.

From inequalities~(\ref{eq:x34ineq}--\ref{eq:t1ineq}) we deduce
\begin{equation} \label{eq:23expEq}
\begin{split}
&(z_{234}x_{23}x_{234})^2(z_{34}x_{12}x_{123}x_{124}x_{1234})(x_{12}x_{13}x_{24}x_{34}x_{123}x_{124}x_{134}x_{1234})(x_{12}x_{123}x_{1234}) = \\
&\frac{z_{34}x_{34}}{x_{14}x_{24}x_{124}}\bigg(\frac{z_{234}x_{23}x_{24}x_{234}}{x_{13}x_{14}x_{134}}\bigg)^2(x_{12}x_{13}x_{14}x_{123}x_{124}x_{134}x_{1234})^3 \ll \frac{n^6}{m^3n_1^3n_2^2n_3d_{34}d_{234}^2}\ll \frac{n^6}{m^3}.
\end{split}
\end{equation} 

Note that any of the factors in parentheses on the left hand side of this inequality, except for the factor $(x_{12}x_{123}x_{1234})$ is a product of parameters constituting one of the defining sets in Lemma~\ref{lem:definingSetLemma}. After distributing the exceptional factor among the others, we see that we have $4$ factors left and that at least one of them is bounded in size by $\mathcal{O}\big(\frac{n^{\nicefrac{3}{2}}}{m^{\nicefrac{3}{4}}}\big)$. Once the bounded factor is fixed we have at most of order $n^{\varepsilon}$ many choices for the corresponding defining set and thus an upper bound of order $\mathcal{O}\big(n^{\varepsilon}\frac{n^{\nicefrac{3}{2}}}{m^{\nicefrac{3}{4}}}\big)$ for the number of choices for all parameters.

Finally, to prove the upper bound of order $n^{\varepsilon}\big(\frac{n^{\nicefrac{8}{5}}}{m}\big)$, from inequalities~(\ref{eq:x23ineq}-\ref{eq:t1ineq}) we derive
\begin{equation}\label{eq:factoringineq}
\begin{split}
&(z_{23}z_{234})(z_{234}x_{23}x_{24})(z_{34}x_{12}x_{123}x_{124}x_{1234})(x_{12}x_{13}x_{14}x_{23}x_{123}x_{124}x_{134}x_{234}x_{1234})^2\times \\
&(x_{12}^2x_{123}^2x_{124}x_{1234}^2) = \frac{z_{23}x_{23}}{x_{13}x_{34}x_{134}}\frac{z_{34}x_{34}}{x_{14}x_{24}x_{124}}\left(\frac{z_{234}x_{23}x_{24}x_{234}}{x_{13}x_{14}x_{134}}\right)^2\times \\
&(x_{12}x_{13}x_{14}x_{123}x_{124}x_{134}x_{1234})^5 \ll \frac{n^9}{m^5n_1^5n_2^3n_3d_{23}d_{34}d_{234}^2} \ll \frac{n^8}{m^5}.
\end{split}
\end{equation}

For the last inequality we note that by definition we have $d_{23}=\prod_{p \in \mathbb{P}}p^{\nu_p(n)-\max\{\nu_p(n_2),\nu_p(n_3)\}}$, where $\nu_p$ denotes the $p$-adic valuation. Hence,
\begin{equation*}
n_2n_3d_{23} = \prod_{p\in \mathbb{P}}p^{\nu_p(n_2)+\nu_p(n_3)+\nu_p(n)-\max\{\nu_p(n_2),\nu_p(n_3)\}}\geq \prod_{p\in \mathbb{P}}p^{\nu_p(n)}= n.
\end{equation*}

By Lemma~\ref{lem:definingSetLemma} any of the factors in parentheses on the very left hand side of~\eqref{eq:factoringineq}, with exception of the factor $(x_{12}^2x_{123}^2x_{124}x_{1234}^2)$, is a product of parameters forming a defining set. Hence, if we fix any of these factors, by Lemma~\ref{lem:divisorBound} we have at most $\mathcal{O}_{\varepsilon}(n^{\varepsilon})$ choices for the corresponding defining set, and thus also at most $\mathcal{O}_{\varepsilon}(n^{\varepsilon})$ choices for all relative greatest common divisors. After distributing the variables of the exceptional factor among the other ones, we conclude that at least one of the remaining factors is bounded from above by $\mathcal{O}\big(\frac{n^{\nicefrac{8}{5}}}{m}\big)$ which gives an upper bound of $\mathcal{O}_{\varepsilon}\big(n^{\varepsilon}\frac{n^{\nicefrac{8}{5}}}{m}\big)$ for the number of solutions of~\eqref{eq:initial4UFeq} altogether.

It may seem a bit mysterious how the equations~\eqref{eq:23expEq} and~\eqref{eq:factoringineq} were found. In Section~\ref{sec:comp}, we describe how we used a computer programme to list many suitable inequalities of this type based on a precomputed list of defining sets. From a list of given inequalities we have chosen the best ones we found. 

\section{Upper bounds on sums of $k \geq 5$ unit fractions} \label{sec:UF5}

In this section, we prove Theorem~\ref{thm:kgeq5UFThm}. We do so by applying a lifting method by Browning and Elsholtz~\cite{BrowningElsholtz} to the result in Theorem~\ref{thm:4UFThm}.

We first derive the bound on $f_5(m,n)$ by summing our upper bound from Theorem~\ref{thm:4UFThm} over several choices of the smallest denominator $a_1$ in the decomposition. Here, we will only consider the bound $f_4(m,n) \ll_{\varepsilon} n^{\varepsilon}\frac{n^{\nicefrac{8}{5}}}{m}$. The reason for this is, that summing over the bound $f_4(m,n) \ll_{\varepsilon} n^{\varepsilon}\frac{n^{\nicefrac{3}{2}}}{m^{\nicefrac{3}{4}}}$ leads to worse upper bounds for $f_5(m,n)$ because the exponent of $m$ is too small. 

In particular, for given $a_1 \in \mathbb{N}$, we consider decompositions of $\frac{m}{n}-\frac{1}{a_1}=\frac{ma_1-n}{na_1}$ as a sum of four unit fractions. We set $ma_1-n =u$, and with the trivial bounds 
$$\frac{n}{m}<a_1\leq \frac{5n}{m},$$
we have
\begin{align*}
f_5(m,n) &\leq \sum_{0<u\leq 4n}f_4\bigg(u,n\frac{u+n}{m}\bigg)\ll_{\varepsilon}n^{\varepsilon}\sum_{0<u\leq 4n}\frac{\big(n\frac{u+n}{m}\big)^{\nicefrac{8}{5}}}{u}\\
&\ll_{\varepsilon} n^{\varepsilon}\bigg(\frac{n^2}{m}\bigg)^{\nicefrac{8}{5}}\sum_{0<u\leq 4n}\frac{1}{u}\ll_{\varepsilon}n^{\varepsilon}\bigg(\frac{n^2}{m}\bigg)^{\nicefrac{8}{5}}.
\end{align*}
We next use~\cite{ElsholtzPlanitzer}*{Lemma C} which summarizes the procedure used in~\cite{BrowningElsholtz}*{Section 4} to lift this upper bound on $f_5(m,n)$ to $f_k(m,n)$ for $k>5$. We give this result here as the following Lemma~\ref{lem:EBLifting}.
\begin{customlem}B \label{lem:EBLifting}
Suppose that there exists $c>1$ such that
$$f_5(m,n) \ll_{\varepsilon} n^{\varepsilon}\bigg(\frac{n^2}{m}\bigg)^{c}.$$
Then for any $k \geq 5$ we have
$$f_k(m,n) \ll_{\varepsilon} (kn)^{\varepsilon}\bigg(\frac{k^{\nicefrac{4}{3}}n^2}{m}\bigg)^{c2^{k-5}}.$$
\end{customlem}
Lemma~\ref{lem:EBLifting} together with our bound on $f_5(m,n)$ above proves Theorem~\ref{thm:kgeq5UFThm}.

\section{Computational aspects} \label{sec:comp}

Here, we describe how we found the proof of Theorem~\ref{thm:4UFThm}. To find inequalities of the type~\eqref{eq:23expEq} and~\eqref{eq:factoringineq} we used a computer algebra system. As stated earlier there are two stages at which computational aspects came into play, the first of which was finding many defining sets. Here we used $96$ equations of type (\ref{eq:systemeq1}-\ref{eq:systemeq6}). For subsets $S_i$, $0 \leq i \leq l$ of the set $\{x_J:J\subset \{1,2,3,4\}, |J|\geq 2\}\cup\{z_J:J\subset\{1,2,3,4\}, 2\leq |J|\leq 3\}$ any of these equations is of the form
\begin{equation} \label{eq:generaldefiningSetEqType}
c_0\prod_{p_J \in S_0}p_J = \sum_{i=1}^l c_i\prod_{p_J \in S_i}p_J,
\end{equation}
where the $c_i$, $0 \leq i \leq l$, are constants depending at most on $m$ and the pattern $(n_1, \ldots, n_4)$. In particular, Lemma~\ref{lem:divisorBound} tells us, that once we fix the parameters in the sets $S_1, \ldots, S_l$, we have at most of order $n^{\varepsilon}$ choices for the parameters in the set $S_0$. 

For a given subset $S$ of parameters we can now go through our $96$ equations and check whether for one of these
\begin{equation} \label{eq:UnionDefiningSetsEquation}
\bigcup_{i=1}^lS_i \subset S.
\end{equation}
If this is the case we add the parameters in $S_0\backslash S$ to $S$ and repeat the process. 

If at some point equation~\eqref{eq:UnionDefiningSetsEquation} does not yield any new parameters for any of the $96$ equations we stop the process. If the set of parameters we obtained in this fashion is the set of all parameters then the original set $S$ was a defining set.

It remains to discuss which equations of the form~\eqref{eq:generaldefiningSetEqType} our program used to find defining sets. We set $I=\{1,\ldots, 4\}$ and we consider the following $8$ types of equations.

\begin{enumerate}
\item The first type of equation arises from considering two of the relative greatest common divisors unknown. In this case equation~\eqref{eq:initial4UFeqmult} may be rearranged such that it factors in one of the following forms:
\begin{align*}
(C_1x_J-C_3)(C_1x_K-C_2)&=C_1C_4+C_2C_3 \\
x_J(C_5+C_6x_K)&=C_7,
\end{align*}
where $J, K \subset I$. This leads to $55$ equations.

\item Next, for $1\leq i<j\leq 4$ and $\{k,l\}=I\backslash \{i,j\}$, we define the integer parameters $z_{ij}$ in~\eqref{eq:zijdef} in a general way:
$$z_{ij}=\frac{\frac{n}{n_id_{ij}}\prod_{i \not\in J, j \in J}x_J+\frac{n}{n_jd_{ij}}\prod_{i\in J, j \not\in J}x_J}{x_{ij}x_{kl}}.$$
From this equation we see that fixing the parameters in the set 
$$\big\{x_J:J\subset\{1,2,3,4\}, (i \in J \wedge j \not \in J) \vee (i \not \in J \wedge j \in J), J\neq \{k,l\} \big\}$$
leads to at most of order $n^{\varepsilon}$ choices for $z_{ij}$ and $x_{ij}$ and, after multiplying with the denominator on the right hand side, to $6$ equations of type~\eqref{eq:generaldefiningSetEqType}. 

\item In addition to the equations corresponding to the parameters $z_{123}$, $z_{134}$ and $z_{234}$ in~\eqref{eq:zijkdef}, we used
$$z_{124}=\frac{\frac{n}{n_1d_{124}}x_{23}x_{24}x_{34}x_{234}+\frac{n}{n_2d_{124}}x_{13}x_{14}x_{34}x_{134}+\frac{n}{n_4d_{124}}x_{12}x_{13}x_{23}x_{123}}{x_{12}x_{14}x_{24}x_{124}}.$$
To get an equation of the form~\eqref{eq:generaldefiningSetEqType} we multiply with the denominator on the right hand side. 

\item Using the definition of $z_{ijk}$, $z_{ij}$, $z_{ik}$ and $z_{jk}$ and setting $l$ to be the single element in $I\backslash \{i,j,k\}$, we have
\begin{align*}
z_{ijk}x_{ij}x_{ik}x_{jk}x_{ijk}&=\frac{d_{ij}}{d_{ijk}} z_{ij}x_{kl}+\frac{n}{n_kd_{ijk}}x_{ij}x_{il}x_{jl}x_{ijl} \\
&=\frac{d_{ik}}{d_{ijk}} z_{ik}x_{jl}+\frac{n}{n_jd_{ijk}}x_{ik}x_{il}x_{kl}x_{ikl} \\
&=\frac{d_{jk}}{d_{ijk}} z_{jk}x_{il}+\frac{n}{n_id_{ijk}}x_{jk}x_{jl}x_{kl}x_{jkl}.
\end{align*}
This leads to twelve equations of type~\eqref{eq:generaldefiningSetEqType}.

\item By definition of the parameters $z_{ijk}$ we may write down the general form of equations \eqref{eq:systemeq1} and \eqref{eq:systemeq2}:
$$m\prod_{\substack{J \subset I \\ l \in J}}x_J = d_{ijk}z_{ijk}+\frac{n}{n_l},$$
where $l$ is the single element in the set $I \backslash \{i,j,k\}$. This leads to $4$ equations and we get that fixing the parameter $z_{ijk}$ leads to at most of order $n^{\varepsilon}$ choices for the parameters in the set
$$\big\{x_J: J \subset I, l\in J\big\}.$$

\item Using just the definition of the $z_{ij}$, we derive $6$ equations of the following form:
$$m\prod_{\substack{J \subset I \\ J \neq \{i,j\}}}x_J = d_{ij}x_{kl}z_{ij}+\frac{n}{n_k}\prod_{\substack{J \subset I \\ k \not\in J \\ J \neq \{i,j\}}}x_J+\frac{n}{n_l}\prod_{\substack{J \subset I \\ l \not\in J \\ J \neq \{i,j\}}}x_J,$$
where $\{k,l\} = I \backslash \{i,j\}$. Hence, once we fix the parameters $z_{ij}$, $x_{kl}$ and those in the set
$$\big\{x_J: J \subset I, k\not\in J, J\neq \{i,j\}\big\} \cup \big\{x_J: J \subset I, l\not\in J, J\neq \{i,j\}\big\},$$
we have at most of order $n^{\varepsilon}$ choices for all the remaining relative greatest common divisors with the exception of $x_{ij}$.

\item Let $\{i,j\}$ and $\{k,l\}$ be a partition of $I$. Again, just using the definition of $z_{ij}$ and $z_{kl}$, we derive from~\eqref{eq:initial4UFeqmult}:
$$m\prod_{\substack{J \subset I \\ J \not\in \{\{i,j\},\{k,l\}\}}}x_J=d_{ij}z_{ij}+d_{kl}z_{kl}.$$
Thus, once we fix the parameter $z_{ij}$ and $z_{kl}$ we have at most of order $n^{\varepsilon}$ choices for all relative greatest common divisors except $x_{ij}$ and $x_{kl}$. We have $3$ equations of this type.

\item Finally let $J_1,J_2 \subset I$ with $J_1 \neq J_2$, $|J_1|=|J_2|=3$ and $J_1 \cap J_2 =\{i,j\}$, $\{k,l\} = I \backslash \{i,j\}$. Then by multiplying $z_{J_1}$ and $z_{J_2}$ we get $6$ equations of the form
$$z_{J_1}z_{J_2}=\frac{n^2}{n_kn_ld_{J_1}d_{J_2}}+\frac{n^2d_{\{i,j\}}}{d_{J_1}d_{J_2}}z_{\{i,j\}}x_{\{k,l\}}^2x_{\{i,k,l\}}x_{\{j,k,l\}}x_{1234}.$$
Thus, if we fix $z_{\{i,j\}}$, $x_{\{k,l\}}$, $x_{\{i,k,l\}}$, $x_{\{j,k,l\}}$ and $x_{1234}$, we have at most of order $n^{\varepsilon}$ choices for the parameters $z_{J_1}$ and $z_{J_2}$. 
\end{enumerate}
Next we need to multiplicatively combine inequalities of type (\ref{eq:x23ineq}-\ref{eq:x234ineq}) in such a way, that the exponent of $n$ on the (larger) right hand side is small and the set of relative greatest common divisors making up for the (smaller) left hand side may be split into many defining sets. In addition to inequalities (\ref{eq:x23ineq}-\ref{eq:x234ineq}) in our computer search we took into account the following seven inequalities:

\begin{alignat*}{3}
z_{12}x_{12}&\leq \frac{2n}{n_1d_{12}} x_{23}x_{24}x_{234} \qquad z_{123}x_{12}x_{13}x_{123}, &&\leq  \frac{3n}{n_1d_{123}}x_{24}x_{34}x_{234}, \\
z_{13}x_{13} &\leq \frac{2n}{n_1d_{13}}x_{23}x_{34}x_{234} \qquad z_{124}x_{12}x_{14}x_{124},&&\leq \frac{3n}{n_1d_{124}}x_{23}x_{34}x_{234}, \\
z_{14}x_{14} &\leq \frac{2n}{n_1d_{14}}x_{23}x_{34}x_{234} \qquad z_{134}x_{13}x_{14}x_{134},&&\leq \frac{3n}{n_1d_{134}}x_{23}x_{24}x_{234}, \\
z_{24}x_{24} &\leq \frac{2n}{n_2d_{24}}x_{14}x_{34}x_{134}.
\end{alignat*}

After multiplying any number of such inequalities up, we divide by the product of all relative greatest common divisors on the right hand side. To clear the resulting denominator on the new left hand side we use inequality~\eqref{eq:t1ineq} together with the inequalities $t_2\leq \frac{12n^2}{n_2m}$ and $t_3\leq \frac{96n^4}{n_3m^2}$, which we derived in the proof of Lemma~\ref{lem:definingSetLemma}. Note that apart from clearing denominators we can add any number of these three inequalities to our previously selected ones.

Furthermore, we took into account that $n_in_jd_{ij}\geq n$ and $n_in_jn_kd_{ijk} \geq n$ for all $1\leq i,j,k\leq 4$. This may lead to a further reduction in size in terms of $n$ on the right hand side of inequalities constructed as above. However, we cannot prove that our computer search covered all possible defining sets and all relevant combinations of inequalities. Hence, it may well be that the exponent in Theorem~\ref{thm:4UFThm} can be improved by conducting a more complete search.

\ \\
{ \emph{Acknowledgements.}}
The authors would like to thank the referee for comments on the
  manuscript and
  acknowledge the support of the 
Austrian Science Fund (FWF): W1230, I 4945-N and I 4406-N.

\end{document}